\documentclass[12pt,twoside,reqno]{amsart}
\usepackage{amsmath}
\usepackage{amsfonts}
\usepackage{amssymb}
\usepackage{color}
\usepackage{mathrsfs}
\usepackage{cite}
\usepackage{cleveref}
\usepackage{geometry}
\usepackage{marginnote}%\marginpar{}, \marginnote{}
\usepackage{todonotes}%\todo{}
\allowdisplaybreaks
\newtheorem{theorem}{Theorem}[section]

\newtheorem{lemma}[theorem]{Lemma}
\newtheorem{proposition}[theorem]{Proposition}

\newtheorem{remark}[theorem]{Remark}
\theoremstyle{definition}
\newtheorem{example}[theorem]{Example}
\numberwithin{equation}{section}
\begin{document}
\title{The Fragmentation Equation with Size Diffusion: Small and Large Size behavior of Stationary Solutions} 
\thanks{Partially supported by Deutscher Akademischer Austauschdienst funding programme \textsl{Research Stays for University Academics and Scientists, 2021} (57552334)}
%
%%%%%%%%%%%%%%%%
\author{Philippe Lauren\c{c}ot}
\address{Institut de Math\'ematiques de Toulouse, UMR~5219, Universit\'e de Toulouse, CNRS \\ F--31062 Toulouse Cedex 9, France}
\email{laurenco@math.univ-toulouse.fr}
\author{Christoph Walker}
\address{Leibniz Universit\"at Hannover\\ Institut f\"ur Angewandte Mathematik \\ Welfengarten 1 \\ D--30167 Hannover\\ Germany}
\email{walker@ifam.uni-hannover.de}
%%%%%%%%%%%%%%%%
%	
\date{\today}
%	
%%%%%%%%%%%%%%%%
%%%%%%%%%%%%%%%%
\begin{abstract}
The small and large size behavior of stationary solutions to the fragmentation equation with size diffusion is investigated. It is shown that these solutions behave like stretched exponentials for large sizes, the exponent in the exponential being solely given by the behavior of the overall fragmentation rate at infinity. In contrast, the small size behavior is partially governed by the daughter fragmentation distribution and is at most linear, with possibly non-algebraic behavior. Explicit solutions are also provided for particular fragmentation coefficients. 
\end{abstract}
%%%%%%%%%%%%%%%%
%%%%%%%%%%%%%%%%
%	
\maketitle
%	
%%%%%%%%%%%%%%%%
%     HEADLINES
%%%%%%%%%%%%%%%%
%
\pagestyle{myheadings}
\markboth{\sc{Ph.~Lauren\c{c}ot \& Ch.~Walker}}{\sc{Behavior of stationary solutions to fragmentation-diffusion equation}}
%	
%
%%%%%%%%%%%%%%%%
%%%%%%%%%%%%%%%%
\section{Introduction}\label{Sec1}
%%%%%%%%%%%%%%%%
%%%%%%%%%%%%%%%%

In \cite{FJMOS2003} the fragmentation equation with diffusion
%\begin{align*}
%	\partial_t \phi(t,x) - D \partial_x^2 \phi(t,x) + a(x) \phi(t,x) & = \int_x^\infty a(y) %b(x,y) \phi(t,y)\ \mathrm{d}y\,, \qquad (t,x)\in (0,\infty)^2\,, \\
%	\phi(t,0) & = 0\,, \qquad t\in (0,\infty)\,, \\
%	\phi(0,x) & = \phi^{in}\,, \qquad x\in (0,\infty)\,, 
%\end{align*}
is proposed to predict the growth of ice crystals as the result of the interplay between diffusion and fragmentation. Triggered by the competition between these two mechanisms,  the existence of stationary states satisfying the nonlocal equation
\begin{subequations}\label{SFD0}
\begin{align}
- f''(x) + a(x) f(x) & = \int_x^\infty a(y) b(x,y) f(y)\ \mathrm{d}y\,, \qquad x\in (0,\infty)\,, \label{SFD1} \\
f(0) & = 0\,, \label{SFD2} 
\end{align}
\end{subequations}
is of particular interest, where $f=f(x)$ denotes the (stationary) size distribution function of particles of size $x\in (0,\infty)$, while $a(x)\ge 0$ is the overall fragmentation rate of particles of size $x$ and $b(x,y)\ge 0$ is the daughter distribution function for particles of size $y$ splitting into particles of size $x<y$. The second-order derivative in \eqref{SFD0} reflects size diffusion (with diffusion rate scaled to 1). 
For the particular case 
\begin{equation*}
	a(x)=\mathfrak{a} x^\gamma\,, \quad b(x,y) = \frac{2}{y}\,, \qquad 0<x<y\,, 
\end{equation*} 
with $\gamma\ge 0$, the steady state can be computed explicitly and reveals a good agreement with experimental data when $\gamma=1$ as shown in  \cite{FJMOS2003,MFJLODS2004}.
A comparison of the steady state with the length distribution of $\alpha$-helices of proteins is also reported in \cite{FJMOS2003}. The experimental curves exhibit a peak for small sizes with a power law behavior near zero and a fast decaying tail for large sizes. The above mentioned explicit steady state matches these two features. 

Since the existence of (non-explicit) stationary solutions to the fragmentation equation with diffusion has been established for a rather large class of fragmentation coefficients \cite{Lau2004,LaWa21}, a first question motivated by the findings in \cite{FJMOS2003} is whether explicit solutions can be computed for a broader choice of fragmentation coefficients than the particular choice above. The next result shows that this is indeed the case. Owing to the linearity of \eqref{SFD0}, we use the total mass
$$
M_1(f):=\int_0^\infty x f(x)\,\mathrm{d} x
$$
as a normalization parameter.

%%%%%%%%%%%%%%%%
\begin{proposition}\label{Prop1}
Assume that there are $\mathfrak{a}>0$, $\gamma\ge 0$ and $\nu\in (-2,0]$ such that
\begin{equation}
	a(x) =  \mathfrak{a}x^\gamma\,, \quad b(x,y) = (\nu+2) \frac{x^\nu}{y^{\nu+1}}\,, \qquad 0<x<y\,. \label{In1}
\end{equation}
There is a unique stationary solution $f_{\gamma,\nu}$ to \eqref{SFD0} such that $M_1(f_{\gamma,\nu})=1$. It is given by 
\begin{equation}
	f_{\gamma,\nu}(z) = c_{\gamma,\nu}{ \sqrt{\mathfrak{a}}} z^{(\nu+3)/2} K_{|\nu+1|/(\gamma+2)}\left( \frac{2 \sqrt{\mathfrak{a}}}{\gamma+2} z^{(\gamma+2)/2} \right)\,, \qquad z\in (0,\infty)\,, \label{In2}
\end{equation}
where $c_{\gamma,\nu}$ is a scaling parameter guaranteeing that $M_1(f_{\gamma,\nu})=1$ and $K_\rho$ denotes the modified Bessel function of the second kind with parameter $\rho\ge 0$.
\end{proposition}
%%%%%%%%%%%%%%%%
 
As pointed out above, the solution $f_{\gamma,0}$ is already computed in \cite{FJMOS2003} for the case $\gamma\ge 0$ and $\nu=0$. Note that, if $\gamma=\nu=0$, then $f_{0,0}(z) = z e^{-z}$, since $K_{1/2}(z) = \sqrt{\pi/(2z)} e^{-z}$ by \cite[Equation~10.39.2]{DLMF}. 

In view of the known properties of modified Bessel functions, an interesting outcome of \Cref{Prop1} is the identification of the behavior of the stationary solution $f_{\gamma,\nu}$ for small and large sizes. Specifically, we infer from \cite[Equation~10.25.3]{DLMF} that, as $z\to\infty$, 
\begin{equation}
	f_{\gamma,\nu}(z) \sim c_{\gamma,\nu} \frac{\sqrt{ \sqrt{\mathfrak{a}}\pi(\gamma+2)}}{2} z^{(4+2\nu-\gamma)/4} e^{-2{{ \sqrt{\mathfrak{a}}}} z^{(\gamma+2)/2}/(\gamma+2)} \label{In20}
\end{equation}
and from \cite[Equations~10.30.2 \&~10.30.3]{DLMF} that, as $z\to 0$, 
\begin{equation}
	\begin{split}
		f_{\gamma,\nu}(z) & \sim \frac{c_{\gamma,\nu} \sqrt{\mathfrak{a}}}{2}\, \Gamma\left(\frac{\nu+1}{\gamma+2}\right)\,\left( \frac{\gamma+2}{\sqrt{\mathfrak{a}}}\right)^{(\nu+1)/(\gamma+2)} z \;\;\text{ for }\;\; \nu \in (-1,0]\,, \\
		f_{\gamma,-1}(z) & \sim - c_{\gamma,-1}\sqrt{\mathfrak{a}}\frac{\gamma+2}{2} z \ln{z} \;\;\text{ for }\;\; \nu =-1\,, \\
		f_{\gamma,\nu}(z) & \sim \frac{c_{\gamma,\nu} \sqrt{\mathfrak{a}}}{ 2 }\, \Gamma\left(\frac{\vert\nu+1\vert}{\gamma+2}\right)\,\left( \frac{\gamma+2}{ \sqrt{\mathfrak{a}}}\right)^{\vert\nu+1\vert/(\gamma+2)} z^{\nu+2} \;\;\text{ for }\;\; \nu \in (-2,-1)\,.
	\end{split} \label{In21}
\end{equation}
In particular, the leading order of the behavior of $f_{\gamma,\nu}$ for large sizes is solely determined by the overall fragmentation rate $a$ and features a stretched exponential tail when $a$ is not constant. The influence of $b$ is in fact only retained  in the exponent  $(4+2\nu-\gamma)/4$ of the algebraic factor. In contrast, the small size behavior of $ f_{\gamma,\nu}$ is prescribed by the daughter distribution function $b$ and reflects the singularity of the latter. Observe that, while $f_{\gamma,\nu}$ may vanish at an arbitrary slower rate near zero, it cannot vanish faster than linearly. It is worth emphasizing here that such a dichotomy is already observed for self-similar solutions to the fragmentation equation without diffusion, see \cite{BCG2013, Fil1961, McZi1987}. \\

The observations derived from \Cref{Prop1} provide the guidelines and the impetus to investigate the small and large size behaviors of solutions to \eqref{SFD0} for a broader class of fragmentation coefficients $a$ and $b$, one aim being to figure out whether the behaviors reported in \eqref{In20} and \eqref{In21} have a generic character. 
We provide in this paper several results in that direction. Since their statements require specific assumptions on the fragmentation coefficients $a$ and $b$, we illustrate our findings in the next result on the particular case when $a$ obeys a power law as in \Cref{Prop1}, while $b$ is of self-similar form. This specific choice allows us to have a concise and rather complete statement. We refer to the subsequent sections for more general results derived under more technical assumptions.

%%%%%%%%%%%%%%%%
\begin{theorem}\label{Thm2}
Assume  there are $\gamma\ge 0$ and $\mathfrak{a}>0$ such that the fragmentation rate $a$ satisfies
\begin{equation}\label{In3x}
	a(x) = \mathfrak{a} x^\gamma\,, \qquad x>0\,.
\end{equation}
Moreover, assume that the daughter distribution function $b$ is of self-similar form
\begin{subequations}\label{In3}
\begin{equation}
	b(x,y) = \frac{1}{y} h\left( \frac{x}{y} \right)\,, \qquad 0<x<y\,, 
\end{equation}
where $h\in L_1((0,1),z\mathrm{d}z)$ is a nonnegative function  satisfying
\begin{equation}
\int_0^1 z h(z)\ \mathrm{d}z = 1
\end{equation}
\end{subequations}
and
\begin{equation}
	 \int_0^1 z^m h(z)\ \mathrm{d}z \le \frac{\chi}{m}\,, \qquad m\ge m_0\,, \label{B3}
\end{equation} 
for some $m_0\ge 1$ and  $\chi> 0$. Set $\alpha := (\gamma+2)/2$. There is a  unique solution $f$ to  \eqref{SFD0} with $M_1(f)=1$, see \Cref{Prop2} below, which satisfies:  \vspace{2mm}  
\begin{itemize}
\item [(a)] {\bf Large size behavior:} For each $\mu>(\alpha+\chi-1)/2$, there  is $\kappa_\mu>0$ such that
\begin{equation*}
	f(1) x^{-\gamma/4} e^{- \sqrt{\mathfrak{a}} x^{\alpha}/\alpha} \le f(x) \le \kappa_\mu x^{1+ \alpha+\mu} e^{- \sqrt{\mathfrak{a}} x^{\alpha}/\alpha}\,, \qquad x\ge 1\,.
\end{equation*}    \vspace{.001mm}
\item [(b)]  {\bf Small size behavior:} Either $h\in L_1(0,1)$ and there is $\ell_0>0$ such that 
$$
f(z)\sim \ell_0 z \;\;\text{ as }\;\; z\to 0\,.
$$
Or $h\not\in L_1(0,1)$ and, if there is $\lambda\in [-1,0]$ such that
\begin{equation}
	\begin{split}
	H\left( \frac{z}{y} \right) & \sim y^\lambda H(z) \;\;\text{ as }\;\; z\to 0 \;\;\text{ for all }\;\; y>0\,, \\
	H\left( \frac{z}{y} \right) & \le y^\lambda (y+1) H(z)\,, \qquad 0<z<y\,, 
	\end{split} \label{B4}
\end{equation}
where
\begin{equation*}
	H(z) := \int_0^z y h(y)\ \mathrm{d}y\,, \qquad z\in [0,1]\,,
\end{equation*}
then
\begin{equation*}
	f(z) \sim \Lambda_0 z \int_z^1 \frac{H(y)}{y^2}\ \mathrm{d}y \;\;\text{ as }\;\; z\to 0 \;\;\text{ with }\;\; \Lambda_0 := \int_0^\infty y^{1+\lambda} a(y) f(y)\ \mathrm{d}y\,.
\end{equation*}
\end{itemize}
\end{theorem}
%%%%%%%%%%%%%%%%

 \Cref{Thm2} reveals that the leading order for large sizes is either an exponential for constant rates or a stretched exponential for non-constant power law rates $a$. It is  generic as it does not depend on the daughter distribution function $b$. The latter might play a role in the algebraic factor which we cannot determine without a detailed knowledge thereof. As for the small size behavior we observe that $f$ cannot vanish faster than linearly as $z\to 0$ and that it is only determined by the behavior of the daughter distribution function for small sizes.

As we shall see below, some of these results extend to non-homogeneous fragmentation rates $a$ and arbitrary daughter distribution functions $b$. More precisely, after recalling the existence and uniqueness of a solution $f$ to \eqref{SFD0} in \Cref{Sec2} for a general class of coefficients $a$ and $b$, we establish a handful of basic properties thereof including preliminary behaviors of $f$ for small sizes. In \Cref{Sec3} we deepen this analysis, first addressing the finiteness of negative moments of $f$ under suitable assumptions for general daughter distribution functions $b$. We then identify precisely the small size behavior of $f$ for self-similar daughter distribution functions $b$ as in \eqref{In3}. This provides in particular a proof of \Cref{Thm2}~{\bf (b)}. We also give an example that $f$ need not have an algebraic behavior near zero. We then turn to the large size behavior of $f$ in \Cref{Sec4}, where we assume a power law for the fragmentation rate $a$. On the one hand, we derive a lower bound on $f$ by the comparison principle. On the other hand, we use moment estimates and adapt some arguments from \cite{BCG2013,CCM2011} in order to obtain the upper bound on $f$ exhibiting the (stretched) exponential tail. This yields \Cref{Thm2}~{\bf (a)}. It is worth mentioning that the derivation of exponential bounds for solutions to kinetic equations is a very active field of research nowadays, in particular for Boltzmann equations, see \cite{ACGM2013,GPT2019,GPV2009,PaTa2018,TAGP2018} and the references therein. Finally, we sketch in \Cref{Sec5} the computations leading to the explicit solutions given in \Cref{Prop1}.

%%%%%%%%%%%%%%%%
%%%%%%%%%%%%%%%%
\section{Preliminary Results}\label{Sec2}
%%%%%%%%%%%%%%%%
%%%%%%%%%%%%%%%%

 In this section we recall the existence and uniqueness of a solution to \Cref{SFD0} established in \cite{LaWa21} for rather general  fragmentation coefficients $a$ and $b$. Moreover, we shall derive  some first qualitative properties for this solution under these general assumptions.

For a precise statement of the existence result we introduce the spaces
$$
X_m := L_1((0,\infty),x^m\mathrm{d}x)
$$ 
for $m\in\mathbb{R}$. 
For $f\in X_m$ and $m\in\mathbb{R}$, we define the moment $M_m(f)$ of order $m$ of $f$ by
\begin{equation*}
M_m(f) := \int_0^\infty x^m\ f(x)\ \mathrm{d}x\,.
\end{equation*}
Throughout the paper we assume that the overall fragmentation rate $a$ satisfies
\begin{equation}
	a\in L_{\infty,loc}([0,\infty))\,, \quad a> 0 \;\;\text{ a.e. in }\;\; (0,\infty)\,, \quad \liminf_{x\to\infty} a(x) \in (0,\infty]\,, \label{A1}
\end{equation}
while the daughter distribution function $b$ is a positive measurable function on $(0,\infty)^2$ satisfying
\begin{equation}
	\int_0^y x b(x,y)\ \mathrm{d}x = y\,, \qquad y\in (0,\infty)\,, \label{B1}
\end{equation}
and
\begin{equation}
	\delta_2 := \inf_{y>0} \left\{ 1 - \frac{1}{y^2} \int_0^y x^2 b(x,y)\ \mathrm{d}x  \right\} >0\,. \label{B2}
\end{equation}
These assumptions ensure in particular the existence and uniqueness of a solution to \eqref{SFD0}:

%%%%%%%%%%%%%%%%
\begin{proposition}\label{Prop2}
Assume that the fragmentation coefficients $a$ and $b$ satisfy \eqref{A1}, \eqref{B1}, and \eqref{B2}. Then \Cref{SFD0} has a unique nonnegative solution $f\in C([0,\infty))\cap C^\infty((0,\infty)) $ satisfying
\begin{equation*}
	f \in \bigcap_{m>-1} X_m \,, \quad a f\in \bigcap_{m> -1} X_m\,, \quad f'' \in \bigcap_{m\ge 1} X_m
\end{equation*}
and $M_1(f)=1$. Moreover,
\begin{equation}\label{p9}
	\lim_{z\to\infty} f(z) = \lim_{z\to\infty} z f'(z) = 0\,. 
\end{equation}
\end{proposition}
%%%%%%%%%%%%%%%%

\begin{proof}
The existence and uniqueness of a nonnegative solution $f\in C([0,\infty))\cap C^\infty((0,\infty)) $ to \eqref{SFD0} such that $f$, $f''$, and $af$ all belong to $X_m$ for any $m\ge 1$ and satisfying $M_1(f)=1$ and \eqref{p9} follow from \cite[Theorem~1.5, Proposition~1.6 \&  Lemma~2.1]{LaWa21}. 

Consider next $m\in (-1,1)$. The property $f\in X_m$ is a consequence of the already known integrability properties of $f$ and $f''$ according to \cite[Lemma~2.1]{LaWa21}. Moreover, owing to \eqref{A1},
\begin{equation*}
	M_{m}(af) \le \|a\|_{L_\infty(0,1)} M_m(f) + M_1(af) < \infty\,,
\end{equation*}
so that the function $af$ also belongs to $X_m$.
\end{proof}

 {\bf General assumption:} Throughout the remainder of this paper, we assume that the fragmentation coefficients $a$ and $b$ satisfy \eqref{A1}, \eqref{B1}, \eqref{B2} and  $f$ denotes the unique nonnegative solution $f\in C([0,\infty))\cap C^\infty((0,\infty))$ to \Cref{SFD0} with $M_1(f)=1$ provided in \Cref{Prop2}.\\

We  now derive from \eqref{SFD0} an alternative differential equation satisfied by $f$ which only involves the first derivative of $f$. The proof is based on a conservative formulation of the fragmentation operator as a first order derivative, see, e.g., \cite[Proposition~10.1.2]{BLL2020b}.

%%%%%%%%%%%%%%%%
\begin{proposition}\label{Prop3}
	For $z\in (0,\infty)$, 
	\begin{equation}
		f(z) - z f'(z) = \int_z^\infty a(y) f(y) \int_0^z x b(x,y)\ \mathrm{d}x\mathrm{d}y\,. \label{AEQ1}
	\end{equation}	
Equivalently, for $z\in (0,\infty)$, 
\begin{equation}
	\frac{\mathrm{d}}{\mathrm{d}z}\left( \frac{f(z)}{z} \right) = - \frac{1}{z^2} \int_z^\infty a(y) f(y) \int_0^z x b(x,y)\ \mathrm{d}x\mathrm{d}y\,. \label{AEQ2}
\end{equation}
\end{proposition}
%%%%%%%%%%%%%%%%

\begin{proof}
We sketch the proof for the sake of completeness. Consider $\vartheta\in C_c^\infty(0,\infty)$. We multiply \eqref{SFD1} by $x \vartheta(x)$ and integrate the resulting identity over $(0,\infty)$. Since $af\in X_1$, we may apply Fubini's theorem to obtain
\begin{equation}
	- \int_0^\infty x \vartheta(x) f''(x)\ \mathrm{d}x = \int_0^\infty a(y) f(y) \int_0^y x [\vartheta(x)-\vartheta(y)] b(x,y)\ \mathrm{d}x\mathrm{d}y\,. \label{P1}
\end{equation}
On the one hand, integrating by parts the term on the left-hand side of \eqref{P1}, we find
\begin{align*}
	- \int_0^\infty x \vartheta(x) f''(x)\ \mathrm{d}x & = \int_0^\infty (x \vartheta'(x) + \vartheta(x)) f'(x)\ \mathrm{d}x \\
	& = \int_0^\infty (x f'(x) - f(x)) \vartheta'(x)\ \mathrm{d}x\,.
\end{align*}
On the other hand, using Fubini's theorem again, along with \eqref{B1}, gives
\begin{align*}
	& \int_0^\infty a(y) f(y) \int_0^y x [\vartheta(x)-\vartheta(y)] b(x,y)\ \mathrm{d}x\mathrm{d}y \\
	& \hspace{1cm} = - \int_0^\infty a(y) f(y) \int_0^y x  b(x,y) \int_x^y \vartheta'(z)\ \mathrm{d}z\mathrm{d}x\mathrm{d}y \\
	& \hspace{1cm} = - \int_0^\infty \vartheta'(z) \int_z^\infty a(y) f(y) \int_0^z x b(x,y)\ \mathrm{d}x\mathrm{d}y\mathrm{d}z\,.
\end{align*}
Collecting the above identities leads us to the formula
\begin{equation*}
	\int_0^\infty (z f'(z) - f(z)) \vartheta'(z)\ \mathrm{d}z = - \int_0^\infty \vartheta'(z) \int_z^\infty a(y) f(y) \int_0^z x b(x,y)\ \mathrm{d}x\mathrm{d}y\mathrm{d}z\,,
\end{equation*}
from which we deduce that there is $J\in\mathbb{R}$ such that
\begin{equation}
	z f'(z) - f(z) + \int_z^\infty a(y) f(y) \int_0^z x b(x,y)\ \mathrm{d}x\mathrm{d}y = J\,, \qquad z\in (0,\infty)\,. \label{P2}
\end{equation}
It now follows from \eqref{B1} and the property $af \in X_1$ that
\begin{equation*}
	0 \le \lim_{z\to\infty} \int_z^\infty a(y) f(y) \int_0^z x b(x,y)\ \mathrm{d}x\mathrm{d}y \le \lim_{z\to\infty} \int_z^\infty y a(y) f(y)\mathrm{d}y = 0\,.
\end{equation*}
 Recalling \eqref{p9} we may then take the limit $z\to\infty$ in \eqref{P2} and deduce that $J=0$, thereby completing the proof of \eqref{AEQ1}. Finally, \eqref{AEQ2} is a straightforward consequence of \eqref{AEQ1}.
\end{proof}

We next turn to some monotonicity and positivity properties of $f$.

%%%%%%%%%%%%%%%%
\begin{proposition}\label{Prop4}
	The function $f$ is positive in $(0,\infty)$ and $z\mapsto f(z)/z$ is decreasing on $(0,\infty)$.
\end{proposition}
%%%%%%%%%%%%%%%%

\begin{proof}
	Set $z_0 := \inf\{ z\in (0,\infty)\ :\ f(z)=0\}$ and assume for contradiction that $z_0>0$. Then the continuity of $f$ implies that $f(z_0)=0$ and we infer from \eqref{AEQ2} that $f(z)=0$ for $z\ge z_0$. On the one hand, since $f\in C^1(0,\infty)$, we deduce that $f'(z_0)=0$. On the other hand, $u:=-f\in C([0,z_0])$ satisfies 
	\begin{align*}
		& u'' - a u \ge 0 \;\;\text{ in }\;\; (0,z_0)\,, \quad u(0)=u(z_0)=0\,, \\
		& u(z_0) > u(z)\,, \qquad z\in (0,z_0)\,.
	\end{align*}
Hopf's boundary lemma, see \cite[Lemma~3.4]{GiTr2001}, then implies the contradiction $-f'(z_0)=u'(z_0)>0$. Consequently, $f>0$ in $(0,\infty)$, from which the strict monotonicity of $z\mapsto f(z)/z$ follows  due to~\eqref{AEQ2}.
\end{proof}

Bearing in mind that \Cref{Prop3} ensures that $f$ belongs to $X_m$ for all $m>-1$, we next investigate more precisely the small size behavior of $f$ and first report the following identities.

%%%%%%%%%%%%%%%%
\begin{lemma}\label{Lem5}
	Let $\theta\in [0,1]$. For $\xi>0$, define
	\begin{align*}
		\mathcal{I}_\theta(\xi) & := \int_\xi^\infty a(y) f(y) \int_0^y \left( \frac{\max\{x,\xi\}^{\theta-1} - y^{\theta-1}}{1-\theta} \right) x b(x,y)\, \mathrm{d}x\mathrm{d}y \ge 0 \;\;\text{ for }\;\; \theta\in [0,1)\,, \\
		\mathcal{I}_1(\xi) & := \int_\xi^\infty a(y) f(y) \int_0^y \ln{\left( \frac{y}{\max\{x,\xi\}} \right)} x b(x,y) \, \mathrm{d}x\mathrm{d}y\ge 0\,.
	\end{align*} 
Then $\mathcal{I}_\theta$ is non-increasing on $(0,\infty)$ and
\begin{equation}
	\xi^{\theta-1} f(\xi) + \theta \int_\xi^\infty z^{\theta-2} f(z)\ \mathrm{d}z = \mathcal{I}_\theta(\xi)\,, \qquad \xi>0\,. \label{In4}
\end{equation}
\end{lemma}
%%%%%%%%%%%%%%%%

\begin{proof}
On the one hand, it follows from \eqref{p9} that
\begin{align*}
	- \int_\xi^\infty z^{\theta} \frac{\mathrm{d}}{\mathrm{d}z}\left( \frac{f(z)}{z} \right)\ \mathrm{d}z & = - \bigg[ z^{\theta-1} f(z) \bigg]_{z=\xi}^{z=\infty} + \theta  \int_\xi^\infty z^{\theta-2} f(z)\ \mathrm{d}z \\
	& = \xi^{\theta-1} f(\xi) + \theta  \int_\xi^\infty z^{\theta-2} f(z)\ \mathrm{d}z\,.
\end{align*}
On the other hand, by Fubini-Tonelli's theorem,
\begin{align*}
	& \int_\xi^\infty z^{\theta-2} \int_z^\infty a(y) f(y) \int_0^z x b(x,y)\ \mathrm{d}x\mathrm{d}y\mathrm{d}z \\
	& \hspace{2cm} = \int_\xi^\infty a(y) f(y) \int_\xi^y z^{\theta-2} \int_0^z x b(x,y)\ \mathrm{d}x\mathrm{d}z\mathrm{d}y \\
	& \hspace{2cm} = \int_\xi^\infty a(y) f(y) \int_0^y x b(x,y) \int_{\max\{x,\xi\}}^y z^{\theta-2}\ \mathrm{d}z\mathrm{d}x\mathrm{d}y \\
	& \hspace{2cm} = \mathcal{I}_\theta(\xi)\,.
\end{align*}
Identity \eqref{In4} is now a straightforward consequence of \eqref{AEQ2} and the above two formulas. 

We finally note for $y\ge \xi_2>\xi_1>0$ and $\theta\in [0,1)$ that
\begin{equation*}
	\int_0^y \left( \frac{\max\{x,\xi_1\}^{\theta-1} - y^{\theta-1}}{1-\theta} \right) x b(x,y)\ \mathrm{d}x\ge \int_0^y \left( \frac{\max\{x,\xi_2\}^{\theta-1} - y^{\theta-1}}{1-\theta} \right) x b(x,y)\ \mathrm{d}x \ge 0
\end{equation*}
since $b$ is nonnegative.
Hence, after integrating the above inequality with respect to $y$ over $(\xi_2,\infty)$ and using the nonnegativity of $a$ and $f$,
\begin{align*}
	\mathcal{I}_\theta(\xi_1) \ge \int_{\xi_2}^\infty a(y) f(y) \int_0^y \left( \frac{\max\{x,\xi_1\}^{\theta-1} - y^{\theta-1}}{1-\theta} \right) x b(x,y)\ \mathrm{d}x\mathrm{d}y \ge \mathcal{I}_\theta(\xi_2)\ge 0\,,
\end{align*}
as claimed. A similar argument gives the monotonicity of $\mathcal{I}_1$. 
\end{proof}

A first consequence of \Cref{Lem5} is that the integrability properties of $f$ near zero stated in \Cref{Prop2} cannot be improved in general and that $f$ does not necessarily belong to $X_{-1}$.  We actually derive a necessary and sufficient condition for $f$ to belong to  $X_{-1}$. A similar result is available for self-similar solutions to the fragmentation equation without size diffusion, see \cite{BiTK2018} and \cite[Proposition~10.1.3]{BLL2020b}.

%%%%%%%%%%%%%%%%
\begin{lemma}\label{Lem6}
The function $f$ belongs to $X_{-1}$ if and only if 
\begin{equation}
	\big[(x,y)\mapsto a(y) f(y) x b(x,y) [\ln{(y/x)}]_+ \big] \in L^1((0,\infty)^2)\,. \label{In5}
\end{equation}
More precisely,
\begin{equation*}
	M_{-1}(f) = \int_0^\infty a(y) f(y) \int_0^y x b(x,y) \ln{\left( \frac{y}{x} \right)}\ \mathrm{d}x\mathrm{d}y\,. 
\end{equation*}
\end{lemma}
%%%%%%%%%%%%%%%%

\begin{proof}
Let $\xi>0$. According to \eqref{In4} (with $\theta=1$), 
	\begin{equation}\label{2}
		f(\xi) + \int_\xi^\infty \frac{f(z)}{z} \ \mathrm{d}z = \mathcal{I}_1(\xi) = \int_\xi^\infty a(y) f(y) \int_0^y x b(x,y) \ln{\left( \frac{y}{\max\{x,\xi\}} \right)}\ \mathrm{d}x\mathrm{d}y\,.
	\end{equation}
Assume first \eqref{In5}. Since 
\begin{equation*}
	a(y) f(y) \int_0^y x b(x,y) \ln{\left( \frac{y}{\max\{x,\xi\}} \right)}\ \mathrm{d}x \le a(y) f(y) \int_0^y x b(x,y) \ln{\left( \frac{y}{x} \right)}\ \mathrm{d}x\,,
\end{equation*}
we infer from \eqref{In5} and Lebesgue's convergence theorem that
\begin{equation*}
	\lim_{\xi\to 0} \mathcal{I}_1(\xi) = \int_0^\infty a(y) f(y) \int_0^y x b(x,y) \ln{\left( \frac{y}{x} \right)}\ \mathrm{d}x\mathrm{d}y\,.
\end{equation*}
Recalling the boundary condition \eqref{SFD2}, we are then in a position to take the limit $\xi\to 0$ in \eqref{2} and conclude that $f\in X_{-1}$ and satisfies the identity stated in \Cref{Lem6}.

Conversely, assuming that $f\in X_{-1}$, we infer from \eqref{SFD2}, \eqref{2}, and the monotonicity of $\mathcal{I}_1$ that
\begin{equation*}
	\sup_{\xi>0} \mathcal{I}_1(\xi) = M_{-1}(f)<\infty\,.
\end{equation*}
Fatou's lemma then entails that 
\begin{equation*}
	\int_0^\infty a(y) f(y) \int_0^y x b(x,y) \ln{\left( \frac{y}{x} \right)}\ \mathrm{d}x\mathrm{d}y \le \liminf_{\xi\to 0} \mathcal{I}_1(\xi) \le M_{-1}(f)\,,
\end{equation*}
from which \eqref{In5} follows.
\end{proof}

In particular, \Cref{Lem6} implies that $f$ cannot vanish algebraically at zero when \eqref{In5} is not satisfied, see \Cref{Ex32} below where this is made explicit.

Another consequence of \Cref{Lem5} is that $f$ cannot vanish faster than linearly at zero, so that $f\not\in X_{-2}$.

%%%%%%%%%%%%%%%%
\begin{lemma}\label{Lem7}
	The function $f$ does not belong to $X_{-2}$. More precisely,
	\begin{equation*}
		\lim_{z\to 0} \frac{f(z)}{z} = \sup_{\xi>0} \mathcal{I}_0(\xi) \in (0,\infty]\,.
	\end{equation*}
\end{lemma}
%%%%%%%%%%%%%%%%

\begin{proof}
Let $\xi>0$. In view of \Cref{Lem5} (with $\theta=0$),
\begin{equation*}
	\frac{f(\xi)}{\xi} = \mathcal{I}_0(\xi) = \int_\xi^\infty a(y) f(y) \int_0^y \left( \frac{1}{\max\{x,\xi\}} - \frac{1}{y} \right) x b(x,y)\ \mathrm{d}x\mathrm{d}y\,,
\end{equation*}
so that
\begin{equation*}
	\lim_{\xi\to 0} \frac{f(\xi)}{\xi} = \sup_{\xi>0} \mathcal{I}_0(\xi) \in (0,\infty]\,,
\end{equation*}
thanks to the monotonicity of $\mathcal{I}_0$ and the positivity of $a$, $b$, and $f$,  the latter being due to Proposition~\ref{Prop4}.
\end{proof}

%%%%%%%%%%%%%%%%
%%%%%%%%%%%%%%%%
\section{Small Size Behavior}\label{Sec3}
%%%%%%%%%%%%%%%%
%%%%%%%%%%%%%%%%

We next investigate in more detail the qualitative behavior of the solution $f$ to \Cref{SFD0} for small sizes.  It turns out to be determined predominantly by the daughter distribution function $b$ and requires no further qualitative properties of the fragmentation rate $a$.
In the first part we consider a general distribution function $b$ and subsequently assume a distribution function of self-similar form~\eqref{In3}.

%%%%%%%%%%%%%%%%
%%%%%%%%%%%%%%%%
\subsection{Small Size Behavior: General Daughter Distribution Functions}\label{Sec31}
%%%%%%%%%%%%%%%%
%%%%%%%%%%%%%%%%

Building upon the outcome of \Cref{Lem7}, we first provide a sufficient condition on $b$,  stating that fragmentation produces a finite number of daughter particles, which guarantees that $f$ behaves linearly at zero. 

%%%%%%%%%%%%%%%%
\begin{proposition}\label{Prop31}
Assume that 
\begin{equation}
	N_0 := \sup_{y>0} \int_0^y b(x,y)\ \mathrm{d}x < \infty\,. \label{SB1}
\end{equation}
Then
\begin{equation*}
\ell_0 := \int_0^\infty a(y) f(y) \int_0^y \left( \frac{1}{x} - \frac{1}{y} \right) x b(x,y)\ \mathrm{d}x\mathrm{d}y = \sup_{\xi>0}\mathcal{I}_0(\xi) < \infty
\end{equation*}
and	$f(z) \sim \ell_0 z$ as $z \to 0$.
\end{proposition}
%%%%%%%%%%%%%%%%

\begin{proof}
By \Cref{Prop2}, the function $af$ belongs to $X_0=L_1(0,\infty)$ and we infer from \eqref{SB1} that $\ell_0$ is finite. Moreover, Lebesgue's convergence theorem ensures that
\begin{equation*}
	\ell_0 = \lim_{\xi\to 0} \mathcal{I}_0(\xi) = \sup_{\xi>0}\mathcal{I}_0(\xi) \,.
\end{equation*}
\Cref{Prop31} then readily follows from \Cref{Lem7}.
\end{proof}

Due to the fact that the left-hand side of \eqref{In4} involves two positive terms when $\theta\in (0,1)$, it is less obvious to derive properties on $f$ from \Cref{Lem5}  for such exponents. Nevertheless, an assumption in the spirit of~\eqref{SB1} allows us to get the following information on the small size behavior of $f$:

%%%%%%%%%%%%%%%%
\begin{proposition}\label{Prop32}
Assume that there is $m\in (0,1)$ such that
\begin{equation}
	N_m := \sup_{y>0} \left\{ \frac{1}{y^m} \int_0^y x^m b(x,y)\ \mathrm{d}x \right\} < \infty\,. \label{SB2}
\end{equation}
Then $f\in X_{m-2}$ with
\begin{equation*}
	m(1-m) M_{m-2}(f) = \int_0^\infty a(y) f(y) \int_0^y x \left( x^{m-1} - y^{m-1} \right) b(x,y)\ \mathrm{d}x\mathrm{d}y
\end{equation*}
and
\begin{equation*}
	\lim_{z\to 0} z^{m-1} f(z) = 0\,. 
\end{equation*}
\end{proposition}
%%%%%%%%%%%%%%%%

\begin{proof}
Recall first that $af\in X_m$ by \Cref{Prop2}. Therefore, owing to \eqref{SB2}, the monotonicity of $\mathcal{I}_{m}$ established in \Cref{Lem5}, and Lebesgue's dominated convergence theorem, we obtain that
\begin{equation*}
	\ell_m := \int_0^\infty a(y) f(y) \int_0^y x \left( x^{m-1} - y^{m-1} \right) b(x,y)\ \mathrm{d}x\mathrm{d}y < \infty
\end{equation*}
and
\begin{equation}
\lim_{\xi\to 0} \mathcal{I}_{m}(\xi) = \sup_{\xi>0} \mathcal{I}_{m}(\xi) = \frac{\ell_m}{1-m}\,. 
\label{SB3}
\end{equation}
We now deduce from \eqref{In4} (with $\theta=m$) and \eqref{SB3} that 
\begin{subequations}\label{SB4}
	\begin{equation}
		0 \le \xi^{m-1} f(\xi)  + m \int_\xi^\infty z^{m-2} f(z)\ \mathrm{d}z \le \frac{\ell_m}{1-m} \label{SB4a}
	\end{equation}
	and
	\begin{equation}
		\lim_{\xi\to 0} \left( \xi^{m-1} f(\xi)  + m \int_\xi^\infty z^{m-2} f(z)\ \mathrm{d}z \right) = \frac{\ell_m}{1-m}\,. \label{SB4b}
	\end{equation}
\end{subequations}
A first consequence of \eqref{SB4} and $f\ge 0$  is that $f\in X_{m-2}$ with $m(1-m) M_{m-2}(f)\le \ell_m$ and that
\begin{equation*}
	\lim_{\xi\to 0} \xi^{m-1} f(\xi) = \frac{\ell_m}{1-m} - m M_{m-2}(f)\in [0,\infty)\,.
\end{equation*}
We then observe that the above small size behavior of $f$ only complies with the integrability property $f\in X_{m-2}$ when this limit is zero. This completes the proof.
\end{proof}

We supplement \Cref{Prop32} with a sufficient condition on $b$ preventing $f$ to lie in $X_{m-2}$.

%%%%%%%%%%%%%%%%
\begin{proposition}\label{Prop33}
Assume that there is $m\in (0,1)$ such that the set
\begin{equation}
	\mathcal{S} := \left\{ y\in (0,\infty)\ :\ \int_0^y x^m b(x,y)\ \mathrm{d}x = \infty \right\} \label{SB5}
\end{equation}
has positive measure. Then $f\not\in X_{m-2}$.
\end{proposition}
%%%%%%%%%%%%%%%%

\begin{proof}
Let $n\ge 1$ and $j\ge 1$ be integers and define 
\begin{equation*}
	\mathcal{S}(n,j) := \left\{ y\in (1/j,\infty)\ :\ \int_{1/j}^y x^m b(x,y)\ \mathrm{d}x \ge n \right\}\,.
\end{equation*}
On the one hand,
\begin{equation}
	\mathcal{S}(n,j) \subset \mathcal{S}(n,j+1)\,, \quad j\ge 1\,,  \;\;\text{ and }\;\; \mathcal{S} \subset \mathcal{S}(n) := \bigcup_{j=1}^\infty \mathcal{S}(n,j)\,. \label{SB6}
\end{equation}
On the other hand, for $n\ge 1$ and $j\ge 1$, it follows from \eqref{B1}, \Cref{Prop2}, and the definition of $\mathcal{S}(n,j)$ that
\begin{align*}
	(1-m) \mathcal{I}_{m}(1/j) & = \int_{1/j}^\infty a(y) f(y) \int_0^y \left( \max\{x,1/j\}^{m-1} - y^{m-1} \right) x b(x,y)\ \mathrm{d}x\mathrm{d}y \\
	& \ge \int_{1/j}^\infty a(y) f(y) \int_{1/j}^y x^{m} b(x,y)\ \mathrm{d}x\mathrm{d}y - \int_{1/j}^\infty y^m a(y) f(y) \mathrm{d}y \\
	& \ge \int_{\mathcal{S}(n,j)} a(y) f(y) \int_{ 1/j}^y x^{m} b(x,y)\ \mathrm{d}x\mathrm{d}y - M_m(af) \\
	& \ge n \int_{\mathcal{S}(n,j)} a(y) f(y)\mathrm{d}y - M_m(af)\,.
\end{align*}
Therefore, in view of \eqref{SB6} and \Cref{Prop2},
\begin{align*}
	(1-m) \sup_{\xi>0} \mathcal{I}_{m}(\xi) & \ge  n \lim_{j\to\infty} \int_{\mathcal{S}(n,j)} a(y) f(y)\mathrm{d}y - M_m(af) \\
	& =  n \int_{\mathcal{S}(n)} a(y) f(y)\mathrm{d}y - M_m(af) \\
	& \ge  n \int_{\mathcal{S}} a(y) f(y)\mathrm{d}y - M_m(af)\,.
\end{align*}
Since $a$ and $f$ are positive in $(0,\infty)$ and $\mathcal{S}$ has positive measure, we may let $n\to\infty$ in the above inequality to conclude that
\begin{equation*}
	\sup_{\xi>0} \mathcal{I}_{m}(\xi) = \infty\,.
\end{equation*}
It then follows from \eqref{In4} (with $\theta=m$) and the monotonicity of $\mathcal{I}_m$ that
\begin{equation}
	\lim_{\xi\to 0} \left( \xi^{m-1} f(\xi) + m \int_\xi^\infty z^{m-2} f(z)\ \mathrm{d}z \right) = \infty\,. \label{SB7}
\end{equation} 
Observing that the monotonicity of $z\mapsto f(z)/z$ established in \Cref{Prop4} implies that
\begin{equation*}
	\left( \xi^m - \frac{\xi^m}{2^m} \right) \frac{f(\xi)}{\xi} \le m \int_{\xi/2}^\xi z^{m-1} \frac{f(z)}{z}\ \mathrm{d}z \le m \int_{\xi/2}^\infty z^{m-2} f(z)\ \mathrm{d}z\,,
\end{equation*}
we infer from \eqref{SB7} and the above inequality that
\begin{equation*}
	m \left( 1 + \frac{2^m}{2^m-1} \right) \lim_{\xi\to 0} \int_{\xi/2}^\infty z^{m-2} f(z)\ \mathrm{d}z \ge \lim_{\xi\to 0} \left( \xi^{m-1} f(\xi) + m \int_\xi^\infty z^{m-2} f(z)\ \mathrm{d}z \right) = \infty\,.
\end{equation*}
Hence, $f\not\in X_{m-2}$ and the proof is complete.
\end{proof}

%%%%%%%%%%%%%%%%
%%%%%%%%%%%%%%%%
\subsection{Small Size Behavior: Self-Similar Daughter Distribution Functions}\label{Sec32}
%%%%%%%%%%%%%%%%
%%%%%%%%%%%%%%%%

We next aim at a more precise identification of the small size behavior of $f$ and focus in this section on self-similar daughter distribution functions $b$ in the sense that $b$ satisfies \eqref{In3}. Since the case $h\in L_1(0,1)$ is already studied in \Cref{Prop31}, we consider the complementary case $h\not\in L_1(0,1)$.
 
%%%%%%%%%%%%%%%%
\begin{proposition}\label{Prop34}
 Suppose that
\begin{equation}h\not\in L_1(0,1)\,. \label{H1a} 
\end{equation}
Define 
	\begin{equation*}
		H(z) := \int_0^z y h(y)\ \mathrm{d}y\,, \qquad z\in [0,1], 
	\end{equation*}
and assume that there are two positive and measurable functions $L\ge L_0$ on $(0,\infty)$ such that
\begin{subequations}\label{H1}
\begin{align}
	&H\left( \frac{z}{y} \right) \le L(y) H(z)\,, \qquad y\in (z,\infty)\,, \ z\in (0,1)\,, \label{H1b} \\
	& H\left( \frac{z}{y} \right) \sim L_0(y) H(z) \;\;\text{ as }\;\; z\to 0 \;\;\text{ for all }\;\; y>0\,. \label{H1c}
\end{align}
\end{subequations}
If 
\begin{equation}
	\int_0^\infty y a(y) L(y) f(y)\ \mathrm{d}y < \infty\,, \label{H2}
\end{equation}
then 
\begin{equation}
	f(z) \sim \Lambda_0 z \int_z^1 \frac{H(y)}{y^2}\ \mathrm{d}y \;\;\text{ as }\;\; z\to 0\,, \label{H3}
\end{equation}
where
\begin{equation*}
	\Lambda_0 := \int_0^\infty y a(y) L_0(y) f(y)\ \mathrm{d}y < \infty\,.
\end{equation*}
\end{proposition}
%%%%%%%%%%%%%%%%

%%%%%%%%%%%%%%%%
\begin{remark}\label{Rem20}
	It actually follows from \eqref{H1c} and \cite[Theorem~1.4.1]{BGT1987} that there is $\lambda\in\mathbb{R}$ such that $L_0(y) = y^\lambda$ for $y>0$. Furthermore, since $H(z/y)\le H(z)$ for $y>1>z>0$, we readily deduce from \eqref{H1c} that $L_0(y)\le 1$ for $y>1$. Thus, $\lambda\le 0$.
\end{remark}
%%%%%%%%%%%%%%%%

\begin{proof}
Owing to \eqref{In3}, we infer from \eqref{AEQ2} that, for $z>0$, 
\begin{align}
	\frac{\mathrm{d}}{\mathrm{d}z} \left( \frac{f(z)}{z} \right) & = - \frac{1}{z^2} \int_z^\infty a(y) f(y) \int_0^z \frac{x}{y} h\left( \frac{x}{y} \right)\ \mathrm{d}x\mathrm{d}y \nonumber\\
	& = - \frac{1}{z^2} \int_z^\infty ya(y) f(y) \int_0^{z/y} x h(x)\ \mathrm{d}x\mathrm{d}y \nonumber \\
	& = - \frac{1}{z^2} \int_z^\infty ya(y) f(y) H\left( \frac{z}{y} \right)\ \mathrm{d}y\,. \label{H4}
\end{align}
Note that \eqref{H1}, \eqref{H2}, and Lebesgue's convergence theorem imply
\begin{equation*}
	\lim_{z\to 0} \frac{1}{H(z)} \int_z^\infty y a(y) f(y) H\left( \frac{z}{y} \right)\ \mathrm{d}y = \Lambda_0\,.
\end{equation*}
Combining this property with \eqref{H4} gives
\begin{equation}
	\frac{\mathrm{d}}{\mathrm{d}z} \left( \frac{f(z)}{z} \right) \sim - \Lambda_0 \frac{H(z)}{z^2} = \Lambda_0 \frac{\mathrm{d}}{\mathrm{d}z} \left( \int_z^1 \frac{H(y)}{y^2}\ \mathrm{d}y \right) \;\;\text{ as }\;\; z\to 0\,. \label{H5}
\end{equation}
Since
\begin{align*}
	\int_\xi^1 \frac{H(y)}{y^2}\ \mathrm{d}y & = \int_0^\xi z h(z) \int_\xi^1 \frac{\mathrm{d}y}{y^2}\ \mathrm{d}z + \int_\xi^1 z h(z) \int_z^1 \frac{\mathrm{d}y}{y^2}\ \mathrm{d}z \\
	& = \frac{H(\xi)}{\xi} + \int_\xi^1 h(z)\ \mathrm{d}z - 1 \ge \int_\xi^1 h(z)\ \mathrm{d}z - 1
\end{align*}
by \eqref{In3} and Fubini-Tonelli's theorem, it follows from \eqref{H1a} that $H\not\in L_1((0,1),z^{-2}\mathrm{d}z)$. This property, along with \eqref{H5}, implies \eqref{H3} after integration.
\end{proof}

To illustrate the somewhat abstract outcome of \Cref{Prop34}, we now provide a couple of examples. We begin with the classical case of a non-integrable negative power law \cite{Fil1961,McZi1987}.

\begin{example}\label{Ex31}
Assume that there is $\nu\in (-2,-1]$ such that 
$$
h(z)=(\nu+2) z^\nu\,,\quad z\in (0,1)\,.
$$ 
According to the selected range of $\nu$, $h$ obviously satisfies \eqref{H1} with $L_0(y)=L(y)=y^{-(\nu+2)}$. Since $-\nu-1\ge 0$, \Cref{Prop2} guarantees that $af\in X_{-\nu-1}$ and thus that \eqref{H2} is satisfied. We may then apply \Cref{Prop34} to conclude that, for $\nu=-1$,
$$
f(z)\sim \Lambda_0 z |\ln{z}|\quad \text{as}\quad z\to 0
$$  
while, for $\nu\in(-2,-1)$,
$$
f(z) \sim \frac{\Lambda_0}{|\nu+1|} z^{\nu+2}\quad \text{as}\quad z\to 0\,.
$$  
\end{example}

The previous example shows in particular that the small size behavior \eqref{In21} of the explicit solutions derived in \Cref{Prop1} is generic in the sense that it is valid for arbitrary fragmentation rates $a$. 

We next turn to a particular case which we believe to be of interest as it features a higher singularity at zero than the previous examples as well as a non-algebraic behavior of $f$ at zero.

\begin{example}\label{Ex32}
Let $\theta\in (0,1)$  be fixed and set
\begin{equation}
	h(z) = \theta (1-\ln{z})^{-\theta-1} z^{-2}\,, \qquad z>0\,. \label{H6}
\end{equation}
In particular, 
\begin{equation*}
	\int_0^1 zh(z)\ \mathrm{d}z = 1 \,, \qquad \int_0^1 z \, \vert\mathrm{ln}{z}\vert \, h(z)\ \mathrm{d}z = \infty\,,
\end{equation*}
so that $f\not\in X_{-1}$ in this case according to \Cref{Lem6}. Moreover,
\begin{equation*}
	H(z) = (1-\ln{z})^{-\theta}\,, \qquad z>0\,,
\end{equation*}
and \eqref{H1} is satisfied with $L_0\equiv 1$ and $L(y) = \max\{y,1\}/y$, $y>0$, the latter being a consequence of the inequality
\begin{equation*}
	y \left( 1 + \frac{\ln{y}}{1+|\ln{z}|} \right)^{-\theta} \le \max\{y,1\}\,, \qquad y\in (z,\infty)\,, \ z\in (0,1)\,.
\end{equation*}
Since $af\in X_0\cap X_1$ by \Cref{Prop2},  assumption \eqref{H2} is satisfied and \Cref{Prop34} implies that
\begin{equation*}
	f(z) \sim \Lambda_0 z \int_z^1 \frac{(1-\ln{y})^{-\theta}}{y^2}\ \mathrm{d}y \sim \Lambda_0 (1-\ln{z})^{-\theta} \;\;\text{ as }\;\; z\to 0\,,
\end{equation*}
the second equivalence being derived by L'Hospital's rule. 
\end{example}

We finish off this section with the proof of \Cref{Thm2}~{\bf (b)}. 

\begin{proof}[Proof of \Cref{Thm2}~{\bf (b)}] If $h\in L_1(0,1)$, then \Cref{Thm2}~{\bf (b)} readily follows from \Cref{Prop31}. If $h\notin L_1(0,1)$, then we infer from \eqref{B4} that $h$ satisfies \eqref{H1b} with $L(y)=y^{\lambda}(y+1)$ and \eqref{H1c} with $L_0(y)=y^\lambda$. Since $af$ belongs to $X_{1+\lambda}$ by \Cref{Prop2}, the assumption \eqref{H2} is also satisfied. We may then apply \Cref{Prop34} to complete the proof.
\end{proof}

%%%%%%%%%%%%%%%%
%%%%%%%%%%%%%%%%
\section{Large Size Behavior}\label{Sec4}
%%%%%%%%%%%%%%%%
%%%%%%%%%%%%%%%%

 We now turn to the large size behavior of $f$. In contrast to the small size behavior which is dominated by the properties of the daughter distribution function $b$, the behavior of $f$ for large sizes is determined by the fragmentation rate $a$. For a fragmentation rate satisfying~\eqref{In3x}, we summarize the outcome in the following proposition, which is in accordance with the special case~\eqref{In20}. For its statement we introduce, for $m\in (1,\infty)$,  
\begin{equation}\label{deltam}
	\delta_m := \inf_{y>0}\left\{ 1 - \frac{1}{y^m} \int_0^y x^m b(x,y)\ \mathrm{d}x \right\}
\end{equation}
and recall that $\delta_2>0$ by assumption \eqref{B2}.

%%%%%%%%%%%%%%%%
\begin{proposition}\label{Prop40}
Assume that there are $\gamma\ge 0$ and $\mathfrak{a}>0$ such that $a(x) = \mathfrak{a} x^\gamma$ and set $\alpha := (\gamma+2)/2$. Assume further that there are $\chi\ge 0$ and $m_0\ge 1$ such that
\begin{equation*}
 1 - \frac{\chi}{m} \le \delta_m \le 1 \,, \qquad m\ge m_0\,.  
\end{equation*} 
Given $\mu>(\alpha+\chi-1)/2$, there is $\kappa_\mu>0$ such that
\begin{equation*}
	f(1) x^{-\gamma/4} e^{-\sqrt{\mathfrak{a}} x^{\alpha}/\alpha} \le f(x) \le \kappa_\mu x^{1+\alpha+\mu} e^{-\sqrt{\mathfrak{a}} x^{\alpha}/\alpha}\,, \qquad x\ge 1\,.
\end{equation*}  
\end{proposition}
%%%%%%%%%%%%%%%%

\Cref{Thm2}~{\bf (a)} is then a direct consequence of \Cref{Prop40}, the latter readily following from \Cref{Lem41} (with $a^*=\mathfrak{a}$)   and \Cref{Lem42} (with $a_*=K=\mathfrak{a}$ and $\xi=\gamma$) below. 

\medskip

Actually, the derivation of the lower bound on $f$ provided in \Cref{Prop40} requires only an upper bound on $a$, while the upper bound on $f$ only depends on  a lower bound on $a$, provided $a$ grows at most algebraically. We thus distinguish these cases in the following.

%%%%%%%%%%%%%%%%
%%%%%%%%%%%%%%%%
\subsection{Lower Bound}\label{Sec41}
%%%%%%%%%%%%%%%%
%%%%%%%%%%%%%%%%

 We first derive the stated lower bound on $f$ in \Cref{Prop40} under slightly more general assumptions.

%%%%%%%%%%%%%%%%
\begin{lemma}\label{Lem41}
Assume that there are $\gamma\ge 0$ and $a^*>0$ such that
\begin{equation}
	a(x) \le a^* x^\gamma\,, \qquad x\ge 1\,. \label{LS1}
\end{equation}
Then, setting $\alpha=(\gamma+2)/2$,
\begin{equation*}
	f(x) \ge f(1) x^{-\gamma/4} e^{\sqrt{a^*}x^{\alpha}/\alpha}\,, \qquad x\ge 1\,.
\end{equation*}
\end{lemma}
%%%%%%%%%%%%%%%%

\begin{proof}
Set $\eta := \sqrt{a^*}/\alpha$ and
\begin{equation*}
	\sigma_{\varepsilon}(x) := f(1) x^{-\gamma/4} e^{-\eta x^\alpha} - \varepsilon\,, \qquad x\ge 1\,,
\end{equation*}
for $\varepsilon\in (0,1)$. We note that the choice of $\eta$ and $\alpha$ guarantees that, for $x\ge 1$, 
\begin{align*}
	- \sigma_{\varepsilon}''(x) + a^* x^\gamma \sigma_{\varepsilon}(x) & = f(1) \left[ -\frac{\gamma(\gamma+4)}{16} x^{-(\gamma+8)/4} + \alpha \eta \left( \alpha - \frac{\gamma}{2} - 1 \right) x^{(\gamma-4)/4} \right] e^{-\eta x^\alpha} \\
	& \qquad + f(1) \left( a^* - \alpha^2 \eta^2 \right) x^{3\gamma/4} e^{-\eta x^\alpha} - \varepsilon a^* x^\gamma\\
	& = -  \frac{f(1) \gamma(\gamma+4)}{16} x^{-(\gamma+8)/4} e^{-\eta x^\alpha} - \varepsilon a^* x^\gamma\,,
\end{align*}
so that
\begin{equation*}
	- \sigma_{\varepsilon}''(x) + a^* x^\gamma \sigma_{\varepsilon}(x)  \le 0\,, \qquad x\in (1,\infty)\,. 
\end{equation*}
Now, let $X_\varepsilon>1/\varepsilon$ be such that 
\begin{equation*}
	f(1) X_\varepsilon^{-\gamma/4} e^{-\eta X_\varepsilon^\alpha} \le \varepsilon\,.
\end{equation*}
Then $\sigma_{\varepsilon}(X_\varepsilon)\le 0 \le f(X_\varepsilon)$, while $\sigma_{\varepsilon}(1) = f(1) e^{-\eta} - \varepsilon \le f(1)$. Consequently, since 
\begin{equation*}
	- f''(x) + a^* x^\gamma f(x) \ge -f''(x) + a(x) f(x) \ge 0\,, \qquad x\in (0,\infty)\,,
\end{equation*}
in view of \eqref{SFD1}, \eqref{LS1}, and the non-negativity of $f$, we may apply the comparison principle on $(1,X_\varepsilon)$ to obtain
\begin{equation*}
	f(x) \ge \sigma_{\varepsilon}(x)\,, \qquad x\in [1,X_\varepsilon]\,.
\end{equation*}
We then let $\varepsilon\to 0$ in the above inequality to complete the proof.
\end{proof}

%%%%%%%%%%%%%%%%
%%%%%%%%%%%%%%%%
\subsection{Upper Bound}\label{Sec42}
%%%%%%%%%%%%%%%%
%%%%%%%%%%%%%%%%

We next establish the upper bound stated in \Cref{Prop40}. To this end, we first provide more information on $\delta_m$ defined in \eqref{deltam}.

%%%%%%%%%%%%%%%%
\begin{lemma}\label{Lem8}
The mapping $m\mapsto \delta_m$ is non-decreasing from $(1,\infty)$ onto $(0,1)$. 
%For each $m>1$ the number $\delta_m$ belongs to $(0,1]$,  and the mapping $m\mapsto \delta_m$ is non-decreasing on $(1,\infty)$.
\end{lemma}
%%%%%%%%%%%%%%%%

\begin{proof}
Let us first observe that, if $m\ge 2$ and $y>0$, 
\begin{equation*}
	\frac{1}{y^m} \int_0^y x^m b(x,y)\ \mathrm{d}x \le \frac{1}{y^2} \int_0^y x^2 b(x,y)\ \mathrm{d}x\,, 
\end{equation*}
so that $1-\delta_m \le 1-\delta_2<1$ by \eqref{B2}. Next, for $m\in (1,2)$, we infer from \eqref{B1} and Jensen's inequality that
\begin{align*}
\frac{1}{y^m} \int_0^y x^m b(x,y)\ \mathrm{d}x & = \frac{1}{y^{m-1}} \int_0^y x^{m-1} \frac{xb(x,y)}{y}\ \mathrm{d}x \le \frac{1}{y^{m-1}} \left( \int_0^y x \frac{x b(x,y)}{y}\ \mathrm{d}x \right)^{m-1} \\
& = \left( \frac{1}{y^2} \int_0^y x^2 b(x,y)\ \mathrm{d}x\right)^{m-1}\,.
\end{align*}
Consequently, $1-\delta_m \le (1-\delta_2)^{m-1}<1$. The monotonicity of $\delta_m$ is then a direct consequence of that of $m\mapsto z^m$ on $(1,\infty)$ for any $z\in (0,1)$.
\end{proof}

%%%%%%%%%%%%%%%%
\begin{lemma}\label{Lem42}
Assume that there are $\gamma\ge 0$, $a_*>0$, $\chi> 0$, and $m_0\ge 1$ such that
\begin{equation}
	a(x) \ge a_* x^\gamma \,, \qquad x>0\,, \label{LS2}
\end{equation}
and
\begin{equation}
 1 - \frac{\chi}{ m} \le \delta_m \le 1 \,, \qquad m\ge m_0\,.  \label{LS3}
\end{equation}
Setting $\alpha=(\gamma+2)/2$, assume further that there are $\xi\ge \gamma$ and $K>0$ such that
\begin{equation}
	a(x) \le K x^\xi\,, \qquad x\ge 1\,. \label{LS2b}
\end{equation}
Then, for any $\mu>(\alpha+\chi-1)/2$, there is $\kappa_\mu>0$  such that
\begin{equation*}
	f(x) \le \kappa_\mu x^{1+\alpha+\mu+\xi-\gamma} e^{-\sqrt{a_*} x^{\alpha}/\alpha}\,, \qquad x\ge 1\,.
\end{equation*}
\end{lemma}
%%%%%%%%%%%%%%%%

\begin{proof}
Let $\mu>(\alpha+\chi-1)/2$ and set $\varepsilon_\mu := (2\mu+1-\alpha-\chi)/4>0$. Consider $m_\mu\ge 2$ such that
\begin{equation}
	m_\mu \alpha \ge \mu+1+\chi\,, \qquad \max\left\{ \frac{\chi(\mu+\alpha)}{m_\mu\alpha-\mu} , \frac{\mu(\mu+1)}{m_\mu \alpha} \right\} \le \varepsilon_\mu\,. \label{LS4}
\end{equation}
For $m\ge m_\mu$, we multiply \eqref{SFD1} by $x^{m\alpha-\mu}$ and integrate over $(0,\infty)$ to obtain, thanks to  \Cref{Prop2} and Fubini's theorem,
\begin{align*}
	(m\alpha-\mu) \int_0^\infty x^{m\alpha-\mu-1} f'(x)\ \mathrm{d}x & + \int_0^\infty x^{m\alpha-\mu} a(x) f(x)\ \mathrm{d}x \\
	& \qquad = \int_0^\infty a(y) f(y) \int_0^y x^{m\alpha-\mu} b(x,y)\ \mathrm{d}x\mathrm{d}y\,.
\end{align*}
Integrating once more by parts, we further obtain
\begin{align*}
	- (m\alpha-\mu) (m\alpha-\mu-1) \int_0^\infty x^{m\alpha-\mu-2} f(x)\ \mathrm{d}x & + \int_0^\infty x^{m\alpha-\mu} a(x) f(x)\ \mathrm{d}x \\
	& \qquad \le (1-\delta_{m\alpha-\mu}) \int_0^\infty y^{m\alpha-\mu} a(y) f(y) \mathrm{d}y\,.
\end{align*}
Hence, 
\begin{equation*}
	\delta_{m\alpha-\mu} \int_0^\infty x^{m\alpha-\mu} a(x) f(x)\ \mathrm{d}x \le (m\alpha-\mu) (m\alpha-\mu-1) \int_0^\infty x^{m\alpha-\mu-2} f(x)\ \mathrm{d}x\,,
\end{equation*}
which gives, together with \eqref{LS2} and the property $m\alpha - 2 = (m-2)\alpha + \gamma$, 
\begin{equation}
	a_* \delta_{m\alpha-\mu} \mathfrak{M}_m \le (m\alpha-\mu) (m\alpha-\mu-1) \mathfrak{M}_{m-2}\,,\qquad m\ge m_\mu\,, \label{LS5}
\end{equation} 
where
$$
\mathfrak{M}_m:= \int_0^\infty x^{m\alpha+\gamma-\mu} f(x)\ \mathrm{d}x\,.
$$
Now, set $\eta := \sqrt{a_*}/\alpha$ and consider $N\ge \max\{m_\mu+2,(\chi+\varepsilon_\mu)/\alpha\}$. On the one hand, we infer from \eqref{LS4} that
\begin{align*}
	\frac{(m\alpha-\mu)(m\alpha-\mu-1)}{m(m-1)\alpha^2} & = 1 - \frac{2\mu+1-\alpha}{(m-1)\alpha} + \frac{\mu(\mu+1)}{m(m-1)\alpha^2} \\
	& \le 1 - \frac{2\mu+1-\alpha-\varepsilon_\mu}{(m-1)\alpha} 
\end{align*}
 for $m\ge m_\mu$, so that
\begin{align*}
	R & := \sum_{m=m_\mu}^N \frac{(m\alpha-\mu) (m\alpha-\mu-1)}{m!} \eta^m \mathfrak{M}_{m-2} \\
	& = a_* \sum_{m=m_\mu}^N \frac{(m\alpha-\mu) (m\alpha-\mu-1)}{m(m-1)\alpha^2} \frac{\eta^{m-2}}{(m-2)!} \mathfrak{M}_{m-2} \\
	& \le  a_* \sum_{m=m_\mu}^N \frac{\eta^{m-2}}{(m-2)!} \mathfrak{M}_{m-2} - a_* \sum_{m=m_\mu}^N \frac{2\mu+1-\alpha-\varepsilon_\mu}{(m-1)\alpha} \frac{\eta^{m-2}}{(m-2)!} \mathfrak{M}_{m-2} \,.
\end{align*}
Hence, 
\begin{align}
	R & \le a_* \sum_{m=m_\mu-2}^{N-2} \frac{\eta^{m}}{m!} \mathfrak{M}_{m}  - a_* \frac{2\mu+1-\alpha-\varepsilon_\mu}{\alpha} \sum_{m=m_\mu-2}^{N-2} \frac{\eta^{m}}{(m+1)!} \mathfrak{M}_{m} \nonumber\\
	&  \le a_* \sum_{m=m_\mu-2}^{N-2} \frac{\eta^{m}}{m!} \mathfrak{M}_{m}  - a_* \frac{2\mu+1-\alpha-\varepsilon_\mu}{\alpha} \sum_{m=m_\mu}^{N-2} \frac{\eta^{m}}{(m+1)!} \mathfrak{M}_{m}\,. \label{LS6} 
\end{align}
On the other hand, by \eqref{LS3} and \eqref{LS4}, 
\begin{align*}
	\delta_{m\alpha-\mu} & \ge 1 - \frac{\chi}{m\alpha-\mu} = 1 - \frac{\chi}{(m+1)\alpha} - \frac{\chi(\mu+\alpha)}{(m+1)(m\alpha-\mu)\alpha} \\
	& \ge 1 - \frac{\chi+\varepsilon_\mu}{(m+1)\alpha}
\end{align*}
for $m\ge m_\mu$, and
\begin{align*}
	L & := a_* \sum_{m=m_\mu}^N \delta_{m\alpha-\mu} \frac{\eta^m}{m!} \mathfrak{M}_{m}  \ge a_* \sum_{m=m_\mu}^N \frac{\eta^m}{m!} \mathfrak{M}_{m} - a_* \frac{\chi+\varepsilon_\mu}{\alpha} \sum_{m=m_\mu}^N  \frac{\eta^m}{(m+1)!} \mathfrak{M}_{m}\\
	& \ge a_* \sum_{m=m_\mu}^{N-2} \frac{\eta^m}{m!} \mathfrak{M}_{m} - a_* \frac{\chi+\varepsilon_\mu}{\alpha} \sum_{m=m_\mu}^{N-2}  \frac{\eta^m}{(m+1)!} \mathfrak{M}_{m} + a_* \sum_{m=N-1}^{N}  \frac{\eta^m}{m!}\left( 1-\frac{\chi+\varepsilon_\mu}{\alpha(m+1)}\right) \mathfrak{M}_{m}\,.
\end{align*}
Thus, noticing that the last term is nonnegative due to the choice of $N$,
\begin{equation}
L \ge a_* \sum_{m=m_\mu}^{N-2} \frac{\eta^m}{m!} \mathfrak{M}_{m} - a_* \frac{\chi+\varepsilon_\mu}{\alpha} \sum_{m=m_\mu}^{N-2}  \frac{\eta^m}{(m+1)!} \mathfrak{M}_{m}\,. \label{LS7}
\end{equation}
Since $L\le R$ by \eqref{LS5}, it follows from \eqref{LS6} and \eqref{LS7} that
\begin{equation*}
	a_* \frac{2\mu+1-\alpha-\chi-2\varepsilon_\mu}{\alpha} \sum_{m=m_\mu}^{N-2}  \frac{\eta^m}{(m+1)!} \mathfrak{M}_m \le \frac{2 a_*}{\alpha} C_1(\mu)\,,
\end{equation*}
where
\begin{equation*}
	C_1(\mu) := \frac{\alpha}{2} \sum_{m=m_\mu-2}^{m_\mu-1} \frac{\eta^{m}}{m!} \mathfrak{M}_m \,.
\end{equation*}
Observe that $C_1(\mu)$ is finite due to the integrability properties of $f$, see \Cref{Prop2}. Owing to the choice of $\varepsilon_\mu$, we end up with
\begin{equation*}
	\varepsilon_\mu \sum_{m=m_\mu}^{N-2}  \frac{\eta^m}{(m+1)!} \int_0^\infty x^{m\alpha+\gamma-\mu} f(x)\ \mathrm{d}x \le C_1(\mu)\,,
\end{equation*}
from which we deduce that
\begin{equation*}
	\sum_{m=m_\mu+1}^{N-1}  \frac{\eta^m}{m!} \int_0^\infty x^{m\alpha+\gamma-\mu-\alpha} f(x)\ \mathrm{d}x \le \frac{\eta C_1(\mu)}{\varepsilon_\mu}\,. 
\end{equation*}
Therefore,
\begin{equation*}
	\int_1^\infty \sum_{m=0}^{N-1} \frac{\eta^m}{m!} x^{m\alpha+\gamma-\mu-\alpha} f(x)\ \mathrm{d}x \le \int_1^\infty \sum_{m=0}^{m_\mu} \frac{\eta^m}{m!} x^{m\alpha+\gamma-\mu-\alpha} f(x)\ \mathrm{d}x + \frac{\eta C_1(\mu)}{\varepsilon_\mu} =: C_2(\mu)\,,
\end{equation*}
and the right-hand side of the above inequality is finite by \Cref{Prop2} and does not depend on $N$. We may then take the limit $N\to\infty$ and deduce from Fatou's lemma that
\begin{equation}
	\int_1^\infty e^{\eta x^\alpha} x^{\gamma-\mu-\alpha} f(x)\ \mathrm{d}x \le C_2(\mu)\,. \label{LS8}
\end{equation}

To transfer the weighted $L_1$-estimate \eqref{LS8} to a pointwise estimate, we invoke  \eqref{AEQ2} which gives, together with \eqref{B1} and \eqref{LS2b},
\begin{equation*}
	- \frac{\mathrm{d}}{\mathrm{d}z} \left( \frac{f(z)}{z} \right) \le \frac{K}{z^2} \int_z^\infty y^{1+\xi} f(y)\ \mathrm{d}y = \frac{K}{z^2} \int_z^\infty y^{1+\xi+\alpha-\mu-\gamma} e^{-\eta y^\alpha} y^{\gamma-\mu-\alpha} e^{\eta y^\alpha} f(y)\ \mathrm{d}y 
\end{equation*}
for $z\ge 1$. We then infer from \eqref{LS8} and the monotonicity of $y\mapsto y^{1+\xi+\alpha-\mu-\gamma} e^{-\eta y^\alpha}$ on $(z_\mu,\infty)$ that
\begin{equation*}
- \frac{\mathrm{d}}{\mathrm{d}z} \left( \frac{f(z)}{z} \right) \le K  C_2(\mu) z^{\xi+\alpha-\mu-\gamma-1} e^{-\eta z^\alpha}\,, \qquad z\ge z_\mu\,,
\end{equation*}
where
\begin{equation*}
	z_\mu^\alpha := \max\left\{ 1 , \frac{(1+\xi+\alpha-\mu-\gamma)_+}{\eta\alpha}\right\}\,.
\end{equation*}
We integrate the above inequality over $(y,\infty)$ for $y\ge z_\mu$ and use once more the monotonicity of $y\mapsto y^{1+\xi+\alpha-\mu-\gamma} e^{-\eta y^\alpha}$ on $(z_\mu,\infty)$ to obtain
\begin{equation*}
	\frac{f(y)}{y} \le K C_2(\mu) \int_y^\infty \frac{z^{1+\xi+\alpha-\mu-\gamma} e^{-\eta z^\alpha}}{z^2}\ \mathrm{d}z \le K C_2(\mu) y^{1+\xi+\alpha-\mu-\gamma} \frac{e^{-\eta y^\alpha}}{y}\,,
\end{equation*}
and thereby complete the proof.
\end{proof}

%%%%%%%%%%%%%%%%
%%%%%%%%%%%%%%%%
\section{Explicit Solutions}\label{Sec5}
%%%%%%%%%%%%%%%%
%%%%%%%%%%%%%%%%

We finally sketch the proof of \Cref{Prop1} and recall that we consider the case where $a$ and $b$ are explicitly given by
\begin{equation*}
	a(x) = \mathfrak{a}x^\gamma \,,\quad b(x,y) = (\nu+2) \frac{x^\nu}{y^{\nu+1}}\,, \qquad 0<x<y\,, 
\end{equation*}
for some $\mathfrak{a}>0$, $\gamma\ge 0$ and $\nu\in (-2,0]$. With this specific choice of $a$ and $b$, it follows from \eqref{AEQ1} that $f$ solves 
\begin{equation*}
	f(z) - zf'(z) = \mathfrak{a}z^{\nu+2} \int_z^\infty y^{\gamma-\nu-1} f(y)\ \mathrm{d}y\,, \qquad z>0\,.
\end{equation*}
Introducing
\begin{equation}
	P(z) := \int_z^\infty y^{\gamma-\nu-1} f(y)\ \mathrm{d}y\,, \qquad z>0\,, \label{ES1}
\end{equation}
we infer from the above equation and the integrability properties of $f$ that $P$ solves
\begin{equation}
	z P''(z) - (\gamma-\nu) P'(z) - \mathfrak{a}z^{\gamma+1} P(z) = 0\,, \qquad z>0\,, \qquad \lim_{z\to\infty} P(z) =0\,. \label{ES2}
\end{equation}
We introduce another unknown function $Q$ defined by 
\begin{equation}
	z^{\beta} Q\left( \frac{\sqrt{\mathfrak{a}}z^\alpha}{\alpha} \right) := P(z)\,, \qquad z>0\,, \label{ES3} 
\end{equation}
with
\begin{equation}
	\alpha = \frac{\gamma+2}{2} \ge 1\,, \qquad \beta := \frac{1+\gamma-\nu}{2} \ge \frac{1}{2}\,, 
\end{equation}
and substitute \eqref{ES3} into \eqref{ES2} to obtain that $Q$ solves
\begin{equation}
	y^2 Q''(y) + y Q'(y) - \left( y^2 + \left( \frac{\beta}{\alpha} \right)^2 \right) Q(y) = 0\,, \qquad y>0\,, \qquad \lim_{y\to\infty} Q(y) =0\,. \label{ES4}
\end{equation}
That is, $Q$ is a bounded solution to the modified Bessel equation \cite[Equation~10.25]{DLMF}, from which we deduce that there is a positive constant $c>0$ such that $Q(y) = c K_{\beta/\alpha}(y)$ for $y>0$, where $K_\rho$ denotes the modified Bessel function of the second kind with parameter $\rho\ge 0$. Coming back to $P$ and using \eqref{ES3}, we end up with 
\begin{equation*}
	P(z) = c z^{\beta} K_{\beta/\alpha}\left( \frac{\sqrt{\mathfrak{a}}z^\alpha}{\alpha} \right) = c \left(\frac{\alpha}{\sqrt{\mathfrak{a}}}\right)^{\beta/\alpha} \left( \frac{ \sqrt{\mathfrak{a}}z^\alpha}{\alpha} \right)^{\beta/\alpha} K_{\beta/\alpha}\left( \frac{ \sqrt{\mathfrak{a}}z^\alpha}{\alpha} \right) \,, \qquad z>0\,.
\end{equation*}
Since
\begin{equation*}
	\frac{\mathrm{d}}{\mathrm{d}z} \left( z^\rho K_\rho(z) \right) = - z^\rho K_{\rho-1}(z)\,, \qquad (\rho,z)\in [0,\infty)^2\,,
\end{equation*}
by \cite[Equation~10.29.4]{DLMF}, it follows from the explicit formula for $P$ and the chain rule that
\begin{equation*}
	P'(z) = - c \sqrt{\mathfrak{a}} \left(\frac{\alpha}{\sqrt{\mathfrak{a}}}\right)^{\beta/\alpha} \left( \frac{ \sqrt{\mathfrak{a}}z^\alpha}{\alpha} \right)^{\beta/\alpha} K_{(\beta-\alpha)/\alpha}\left( \frac{\sqrt{\mathfrak{a}}z^\alpha}{\alpha} \right) z^{\alpha-1} = - c  \sqrt{\mathfrak{a}} z^{\alpha+\beta-1} K_{(\beta-\alpha)/\alpha}\left( \frac{\sqrt{\mathfrak{a}}z^\alpha}{\alpha} \right)
\end{equation*}
for $z>0$. Recalling that $f(z) = - z^{1+\nu-\gamma} P'(z)$ by \eqref{ES1}, we conclude that 
\begin{equation*}
	f(z) = c \sqrt{\mathfrak{a}}z^{(\nu+3)/2} K_{|\nu+1|/(\gamma+2)}\left( \frac{ \sqrt{\mathfrak{a}}z^\alpha}{\alpha} \right)\,, \qquad z\in (0,\infty)\,,
\end{equation*}
as claimed in \Cref{Prop1}.

%%%%%%%%%%%%%%%%
%%%%%%%%%%%%%%%%
\section*{Acknowledgments}
%%%%%%%%%%%%%%%%
%%%%%%%%%%%%%%%%

This work was done while PhL enjoyed the kind hospitality of the Institut f\"ur Angewandte Mathematik, Leibniz Universit\"at Hannover.	

%%%%%%%%%%%%%%%%
%%%%%%%%%%%%%%%%
\bibliographystyle{siam}
\bibliography{StatSolFragmDiff}
%%%%%%%%%%%%%%%%
%%%%%%%%%%%%%%%%
	
\end{document}